\documentclass[12pt]{article}
\usepackage[top=2.54cm, bottom=2.54cm, left=2.70cm, right=2.70cm]{geometry}
\usepackage[tbtags]{amsmath}
\usepackage{amssymb}
\usepackage{amsthm}
\usepackage{fancyhdr}
\usepackage{latexsym}
\usepackage{mathrsfs}
\usepackage{wasysym}
\usepackage{float}
\usepackage{graphicx}
\usepackage[numbers,sort&compress]{natbib}
\usepackage{setspace}
\usepackage{xcolor}
\allowdisplaybreaks[4]

\newtheorem{theorem}{\bf Theorem}[section]
\newtheorem{lemma}[theorem]{Lemma}

\newtheorem{problem}[theorem]{Problem}

\newtheorem{conj}[theorem]{Conjecture}
\newtheorem{claim}{Claim}
\newtheorem{observation}[theorem]{Observation}

\begin{document}
\begin{spacing}{1.1}
\title{Oriented paths with two blocks in bipartite oriented graphs}
\author{Bin Chen$^a$, \, Meishuang Chen$^b$\thanks {Email address: cmsfzu07@163.com (M. Chen)}, \, Xinmin Hou$^{c,d}$, \, Xinyu Zhou$^e$\\
\small$^a$School of Mathematics and Statistics, Fuzhou University, Fujian, China\\
\small$^b$Center for Discrete Mathematics, Fuzhou University, Fujian, China\\
\small $^c$School of Mathematical Sciences,\\
\small  University of Science and Technology of China, Hefei 230026, Anhui, China\\
\small$^d$ Hefei National Laboratory,\\
\small University of Science and Technology of China, Hefei 230088, Anhui, China\\
\small$^e$ College of Science,\\
\small Southwest Petroleum University, Chengdu 610500, Sichuan, China\\}
\date{}
\maketitle
{\bf Abstract}\,: 
Stein conjectured that for any integer $k\geq 2$, every oriented graph with minimum semidegree greater than $k/2$ contains every orientation of a path with $k$ edges. Recently, Chen, Hou and Zhou proved this conjecture to be true for any oriented path with two blocks, where a block of an oriented path is a maximal directed subpath within it. In this paper, we prove that every bipartite oriented graph with minimum semidegree at least $3k/8+2$ contains every oriented path with two blocks of length $k$ for $k\ge 2$. Moreover, in contrast to the general oriented setting, we highlight that the minimum semidegree threshold in the bipartite setting is closely related to the number of blocks.

\noindent\textbf{AMS subject classifications:} 05C20; 05C38.

\noindent\textbf{Keywords:} Bipartite oriented graph, minimum semidegree, oriented path.

\section{Introduction}

\noindent

One of the most fundamental and popular topics in graph theory is to guarantee the existence of long paths in a graph with degree constraints. This line of research can be traced back to the work of Dirac in 1952 \cite{Dirac1952} in which he determined an optimal bound with respect to the minimum degree that forces a graph to contain a hamiltonian path, i.e., a path that visits every vertex of the graph. For directed graphs\;(or digraphs), the natural analogue is to seek sufficient degree conditions that ensure the existence of oriented paths of specified types and lengths.

The minimum semidegree condition often plays an important role in such problems, where the minimum semidegree  $\delta^0(D)$ of a digraph $D$ is the smallest number between its minimum out-degree and  minimum in-degree. An oriented graph is a digraph containing no directed cycle of length at most two. A classical result of Jackson \cite{Jackson1981} asserts that for any positive integer $k$, every oriented graph of minimum semidegree at least $k$ admits a directed path of length $2k$. Researchers want to understand more relevant information about the relationship between other types of oriented paths and the minimum semidegree. In 2020, Stein \cite{Stein2020} proposed the following conjecture.

\begin{conj}[Stein~\cite{Stein2020}]\label{conj1}
For any positive integer $k$, every oriented graph $D$ with $\delta^0(D)>k/2$ contains any oriented path of length $k$.
\end{conj}

Note first that the above threshold is best possible due to the existence of a disjoint union of regular tournaments on $k$ vertices. Although Conjecture~\ref{conj1} has attracted a great deal of attention from many scholars, it remains widely open. Fortunately, there are a number of results in support of this conjecture. Let us now give several known results regarding Conjecture~\ref{conj1}. From a result of Kelly \cite{Kelly2011} one can derive that Conjecture~\ref{conj1} is true for $k\geq 3n/4+o(n)$ when $n$ is sufficiently large. An \emph{anti-directed path} is an oriented path in which every internal vertex has out-degree 0 or in-degree 0. In \cite{Stein2024}, Stein and Z\'{a}rate-Guer\'{e}n  gave an approximate result for this type of oriented paths. Concretely, they showed that for all $\eta\in (0, 1)$, there exists $n_0$ such that for all $n\ge n_0$ and $k\ge \eta n$, every oriented graph on $n$ vertices with minimum semidegree larger than $(1 +\eta)k/2$ contains every anti-directed path of length $k$.

Both of the above conclusions are true when $n$ is large enough. Let us next present some precise results regarding Conjecture~\ref{conj1}, which is true for all $n$ rather than requiring $n$ to be sufficiently large. It is known that Conjecture~\ref{conj1} is true for directed paths due to Jackson  \cite{Jackson1981}. In 2023, Klimo\u{s}ov\'{a} and Stein \cite{Klimosova2023} proved that for any $k\geq 3$, every oriented graph of minimum semidegree at least $(3k-2)/4$ contains any anti-directed path of length $k$. Subsequently, the lower bound was reduced by Chen, Hou and Zhou \cite{Chen2025} to $(2k+1)/3$, by Skokan and Tyomkyn \cite{Skokan2025} to $5k/8$, as well as by Grzesik and Skrzypczyk \cite{Grzesik2025} to $(k-1+\sqrt{k-3})/2$. Very recently, Gao and Lo \cite{GL} confirmed Conjecture~\ref{conj1} for anti-directed paths.

Apart from directed paths and anti-directed paths, the oriented paths with two blocks are usually another widely studied object because their structures are not particularly complex, where a block of an oriented path is a maximal directed subpath within it. It is worth mentioning that this type of oriented paths have been studied extensively under chromatic number condition, see, e.g., \cite{ElSahili2004,Addario2007,ElSahiliKouider2007}. In the semidegree setting,  Penev, Taruni, Thomassé, Trujillo-Negrete and Tyomkyn \cite{PTTTT}, and Chen, Hou and Zhou \cite{Chen20252} have independently investigated this type of oriented paths for Conjecture~\ref{conj1}. In 2025, Chen, Hou and Zhou~\cite{Chen20252} verified that this conjecture holds for oriented paths with two blocks.

There are also several interesting conclusions in bipartite oriented graphs.
In 1982, Ayel~\cite{Ayel1982} initiated this line of
research by investigating similar problems in general bipartite oriented graphs, and showed that every oriented bipartite graph of minimum semidegree at least $k$ admits either a directed cycle of length at least $4k$ 
or a directed path of length at least $4k+1$. Subsequently, Zhang \cite{Zhang1987} verified that under the same condition there must be either a directed cycle of length at least $4k$ 
or a directed path of length at least $4k+3$. Actually, the mentioned two results are straightforward corollaries of the original results which were studied under Dirac's condition.

Motivated by the above work, we will explore oriented paths with two blocks in bipartite oriented graphs under minimum semidegree conditions. Our main result is as follows.
\begin{theorem}\label{A}
Let $k\ge 2$ be an integer and let $D$ be a bipartite oriented graph
with $\delta^0(D)\ge 3k/8+2$. Then $D$ contains any oriented path with two blocks of length $k$.
\end{theorem}


The rest of this paper is arranged as follows. In the next section, we introduce some notation and present several auxiliary results. The proof of
Theorem~\ref{A} will be provided in Section 3. We conclude this paper with some remarks in the last section.

\section{Preliminaries}

\noindent

For positive integers $\alpha$ and $\beta$, let $[\alpha,\beta]=\{\alpha, \alpha+1, \dots, \beta\}$, particularly, we use $[\alpha]$ to mean $[1,\alpha]$. 
Let $D=(V,A)$ be a digraph.  For an arc $(x,y)\in A(D)$, we write
$x\to y$, and we say $x$ and $y$ to be its \emph{tail} and \emph{head}, respectively, sometimes we also use $xy$ to mean $(x,y)$.  If $(x,y)\notin A(D)$, then we write $x\nrightarrow y$. We use $x\sim y$ to mean $x$ and $y$ are \emph{adjacent}, namely, there is an arc connecting them, and $x\nsim y$ to mean $x$ and $y$ are nonadjacent. Moreover, we use $x\nsim Y$ to mean that $x$ and $y$ are nonadjacent for any $y\in Y$, where $Y$ is a subset of $V(D)$.

The
\emph{out-neighborhood} and \emph{in-neighborhood} of a vertex $x$ are $
 N_D^+(x)=\{y\in V(D):x\to y\}$ and
 $N_D^-(x)=\{y\in V(D):y\to x\}$,
and its \emph{out-degree} and \emph{in-degree} are $d_D^+(x)$ and $d_D^-(x)$.  We omit the subscript $D$ if there is no confusion.  The vertices within $N^{+}_{D}(x)$ and $N^{-}_{D}(x)$ are \emph{out-neighbors} and \emph{in-neighbors} of $x$, respectively. Denote by $\delta^{+}(D)$ and $\delta^{-}(D)$ the \emph{minimum out-degree} and \emph{minimum in-degree} of $D$. Let $\delta^{0}(D)=\text{min}\{\delta^{+}(D),\delta^{-}(D)\}$ be the \emph{minimum semidegree} of $D$. Let us denote $\delta(D)$ to be the \emph{minimum degree} of $D$, which is the minimum value among the sum of $d^{+}(x)$ and $d^{-}(x)$ for any $x\in V(D)$.

A \emph{directed path} $P$ in $D$ is a sequence of $h$ distinct vertices $x_{1},x_{2},\ldots,x_{h}$ such that $x_{i}\rightarrow x_{i+1}$ for any $i\in[h-1]$, and we denote $\ell(P)$ to be its length. The vertex $x_1$
  is called the \emph{initial vertex} of $P$, and 
$x_h$ is called the \emph{terminal vertex} of 
$P$. Similarly, a \emph{directed cycle} $C$ in $D$ is a sequence of $h$ distinct vertices $y_{1},y_{2},\ldots,y_{h}$ such that $y_{i}\rightarrow y_{i+1}$ for any $i\in[h-1]$ and $y_h\to y_1$. For convenience, let $\langle y_1,y_2,\ldots, y_h\rangle$ be such a directed cycle, and let $\ell(C)$ be its length. A hamiltonian directed path\;(resp., cycle) of $D$ is a directed path\;(resp., cycle) passing through all vertices of $D$. Given integers $s$ and $t$, let $P(s,t)$ be an oriented path which consists of $s$ consecutive forward arcs followed by $t$ consecutive backward arcs. Observe that $P(s,t)$ is isomorphic to $P(t,s)$. Symmetrically, we denote $Q(s,t)$ to be an oriented path consisting of $s$ consecutive backward arcs followed by $t$ consecutive forward arcs. Clearly, $Q(s,t)$ is isomorphic to $Q(t,s)$. Obviously, every oriented path with two blocks is isomorphic to $P(s,t)$ or $Q(s,t)$ for some integers $s,t$. For two directed paths $P$ and $Q$ each of which has length at least one with a unique common terminal vertex, we denote $P+Q$ to denote the oriented path obtained by joining them at this common vertex. Notice that $P+Q$ is an oriented path with two blocks of type $P(s,t)$ for some  integers $s$ and $t$.

Recall that an oriented graph is a digraph containing no directed cycles of length at most two. A digraph $D$ is \emph{bipartite} if its vertex set admits a
partition $V(D)=V_1\cup V_2$ such that every arc has one end in $V_1$
and another in $V_2$. For a vertex set $X\subseteq V(D)$, we write
$D[X]$ for the subdigraph induced by $X$, and let
$D-X=D[V(D)\backslash X]$.

Next, we would like to give several auxiliary results. The coming lemma due to Zhang \cite{Zhang1987} tells us one can find a long directed path in bipartite oriented graphs with large minimum semidegree.

\begin{lemma} [\cite{Zhang1987}] \label{longpc}
 Let $D$ be a bipartite oriented graph with $\delta^0(D)\ge k$. Then $D$ contains a  directed path of length $4k-1$.
\end{lemma}


Also, we need the following two simple but very useful observations, and we omit the proofs here.

\begin{observation}\label{obs:endpoints}
Let $D$ be a bipartite oriented graph and let
$P=v_1v_2\ldots v_\ell$ be a longest directed path in $D$. Then $N^-(v_1)\cup N^+(v_\ell)\subseteq V(P)$.
\end{observation}

\begin{observation}\label{sub-path}
Let $s^*,s,t^*,t\geq 0$ be integers with $s^*\geq s$ and $t^*\geq t$. Then every $P(s^*,t^*)$ contains a copy of $P(s,t)$.
\end{observation}

\begin{observation}\label{lem:greedy}
Let $D$ be an oriented graph and let $Q$ be any orientation of a path with
$m\geq 1$ arcs. If $\delta^0(D)\ge m$, then $D$ contains a copy of $Q$.
\end{observation}

We end this section by showing the following lemma which will be frequently used in the proof of {\bf Part I} in Theorem~\ref{A}.

\begin{lemma}\label{lemma}
Let $s,t,k$ be positive integers with $s\geq t$ and $s+t=k$, and let $D$ be a bipartite oriented 
graph with $\delta^0(D)\ge t/2+1
$ such that $V(D)=V_1\cup V_2$.
Let $P=v_1v_2\ldots v_\ell$ be a longest directed path in $D$, and suppose
that $\ell\ge2s+1$.  Denote by
$
 X=\{v_1,\ldots,v_t\}$,
$ Y=\{v_{t+1},\ldots,v_{\ell-t}\}$ and
$ Z=\{v_{\ell-t+1},\ldots,v_\ell\}$.
Then $D$ contains a copy of
$P(s,t)$ if one of the following holds:

$(i).$ $N^-(v_1)\cap Y\ne\emptyset$ or
          $N^+(v_\ell)\cap Y\ne\emptyset$;

$(ii).$ $N^-(v_{i+1})\nsubseteq V(P)$ for some
$v_i\in N^-(v_1)\cap(Z\backslash\{v_\ell\})$;

$(iii).$ $D[V(P)]$ contains a hamiltonian directed cycle.
\end{lemma}

\begin{proof}
Since $s\geq t$ and $s+t=k$, we thus have $s\ge k/2$ and $t\le k/2$.
In addition, the assumption $\ell\ge 2s+1$ gives us that $\ell\ge k+1$.  From Observation~\ref{obs:endpoints} we know
$        N^-(v_1)\cup N^+(v_\ell)\subseteq V(P).
$
In the following, we may assume without loss of generality that $v_1\in V_1$.

\smallskip
\noindent
 $(i).$
Suppose first that  $N^+(v_\ell)\cap Y\neq\emptyset$, and let
$v_h\in N^+(v_\ell)\cap Y$. Denote by
\[
P_1=v_{h+1}\ldots v_\ell v_h \qquad \text{and} 
\qquad
P_2=v_1v_2\ldots v_h.
\]
Since $v_h\in Y$, we have $\ell(P_1)=\ell-h\ge t$ and $ \ell(P_2)=h-1\ge t$.
If $h-1\ge s$, then $\ell(P_2)\ge s$.
Otherwise, $h\le s$, and hence $\ell(P_1)=\ell-h \ge\ell-s \ge s+1.$
This indicates, by Observation~\ref{sub-path}, that $P_1+ P_2$ contains  a copy of $P(s,t)$.

Now suppose that $N^-(v_1)\cap Y\neq\emptyset$, and let $v_i\in N^-(v_1)\cap Y$. By the maximality of $P$, one can see that
$N^-(v_1)\subseteq V(P)\backslash\{v_1\}$. By the previous result, we may assume that $N^+(v_\ell)\cap Y=\emptyset$. Since $\delta^0(D)\ge t/2+1>|Z|/2$, there exists a vertex $v_j\in N^+(v_\ell)\cap X$. Denote by \[
P_1=v_{j+1}\ldots v_i v_1v_2\ldots v_j\qquad \text{and}
\qquad
P_2=v_{i+1}\ldots v_\ell v_j.
\]
As $X\subsetneq V(P_1)$ and $Z\subsetneq V(P_2)$, we can obtain that $\ell(P_1)\ge t$ and $\ell(P_2)\ge t$. If $i\ge s+1$, then
	$\ell(P_1)=i-1\ge s$. Otherwise, $\ell(P_2)=\ell-i\ge2s-s=s$. By Observation~\ref{sub-path}, we conclude that $P_1+ P_2$ contains a copy of $P(s,t)$.

\smallskip
\noindent
$(ii).$  Pick arbitrarily a vertex $v_i\in N^-(v_1)\cap(Z\backslash\{v_\ell\})$ such that $N^-(v_{i+1})\nsubseteq V(P)$. Let
$w_r\in N^-(v_{i+1})\backslash V(P)$, and let $P^*=w_1w_2\ldots w_r$
be a longest directed path in $D-V(P)$ terminating at $w_r$. Let us denote  $C=\langle v_1,v_2,\ldots ,v_i\rangle$. Since $v_i\in Z$, we can deduce that $\ell(C)>|X\cup Y|
=\ell-t
\ge2s+1-s
\ge s+1$.

When $r\ge t$, we denote \[
P_1=w_1w_2\ldots w_rv_{i+1}\qquad \text{and}
\qquad
P_2=v_1v_2\ldots v_{i+1}.
\]
Then $\ell(P_1)\ge t$ and
$\ell(P_2)=\ell(C)\ge s+1$.
It follows from Observation~\ref{sub-path} that
$P_1+ P_2$ contains  a copy of $P(s,t)$.

It remains to consider that $r<t$. By the maximality of $P^*$, we see that
$N^-(w_1)\subseteq V(P)\cup(V(P^*)\backslash\{w_1\})$. Suppose first that
$N^-(w_1)\cap V(C)\neq\emptyset$, and let
$v_a\in N^-(w_1)\cap V(C)$.
If $a\neq i$, then we define
\[
P^{**}=v_{a+1}v_{a+2}\ldots v_iv_1v_2\ldots v_a
w_1w_2\ldots w_r
v_{i+1}\ldots v_\ell.
\] If $a=i$, then we define
\[
P^{**}=v_1v_2\ldots v_i
w_1w_2\ldots w_r
v_{i+1}\ldots v_\ell.
\]
It is not hard to check that  $\ell(P^{**})>\ell(P)$,
contradicting the fact that $P$ is a longest directed path in $D$.
Therefore, we have $N^-(w_1)\cap V(C)=\emptyset$, which implies that
$N^-(w_1)\subseteq
(V(P)\backslash V(C))
\cup
(V(P^*)\backslash\{w_1\})$.

Recall that $D$ is a bipartite oriented graph. There are at most $\lfloor r/2\rfloor$ vertices in $V(P^*)\backslash\{w_1\}$  belonging to different part from where $w_1$ is located. Thereby, we can derive that 
$$
|N^-(w_1)\cap\{v_{i+1},\ldots,v_\ell\}|
\ge \delta^0(D)-\left\lfloor\frac r2\right\rfloor
\ge\lceil t/2+1\rceil
-\left\lfloor\frac{t-1}{2}\right\rfloor\ge2,$$ implying that
 $w_1$ has at least two in-neighbors in $\{v_{i+1},v_{i+2},\ldots,v_\ell\}$.
Let $v_j$ be a vertex with largest subscript in $N^-(w_1)$. Notice that $i+2\le j\le\ell$.

Denote by
\[
P_1=v_1v_2\ldots v_{i+1}\qquad \text{and}
\qquad
P_2=v_{i+2}\ldots
v_jw_1w_2\ldots w_rv_{i+1}.
\]

It is clear that
$\ell(P_2)\ge
2|N^-(w_1)|
\ge2\delta^0(D)
\ge t+2$.
Moreover,
$\ell(P_1)\ge \ell(C)>s$.
Hence, $P_1+ P_2$ contains a copy of $P(s,t)$ by Observation~\ref{sub-path}.

\smallskip
\noindent
$(iii).$
Finally, suppose that
$D[V(P)]$
contains a hamiltonian directed cycle.
Without loss of generality, we may assume that this cycle is $\langle v_1,v_2,\ldots, v_\ell \rangle$. By Observation~\ref{obs:endpoints},
 every in-neighbor of $v_1$ belongs to  $V(P)$. Accordingly, we have
$$N^+(v_\ell)\cup N^-(v_\ell)
	\subseteq
	V(P)\backslash\{v_\ell\}
	=
	X\cup Y\cup Z\backslash\{v_\ell\}.$$
On the other hand, one can infer that
$$|N^+(v_\ell)\cup N^-(v_\ell)|
	\ge
	2\delta^0(D)\ge t+2 \ge \left\lceil
	|X\cup Z\backslash\{v_\ell\}|/2
	\right\rceil+2,$$
since $D$ is a bipartite oriented graph. 
It is easily seen that
$|(N^+(v_\ell)\cup N^-(v_\ell))\cap Y|\ge2$.
Let $b$ be the smallest subscript such that
$v_b\in (N^+(v_\ell)\cup N^-(v_\ell))\cap Y$.
By $(i)$, we may assume that
$v_b\in N^-(v_\ell)$.
Denote by
\[
P_1=v_{b+1}v_{b+2}\ldots v_\ell\qquad \text{and}
\qquad
P_2=v_1v_2\ldots v_bv_\ell.
\]
Consequently, we can also derive from Observation~\ref{sub-path} that
$P_1+ P_2$ contains  a copy of $P(s,t)$.
\end{proof}

\section{Proof of Theorem~\ref{A}}

\noindent

Let $D$ be an arbitrary bipartite oriented graph with bipartite sets $V_1$ and $V_2$ of minimum semidegree at least $3k/8+2$, where $k\geq 2$. Our goal is to verify that $D$ contains every oriented path with two blocks of length $k$. Let us briefly demonstrate that it is sufficient to prove that $D$ contains a copy of $P(s,t)$ for any $s,t$ with $s+t=k$. Actually, assume the statement is true, by reversing every arc of $D$ we know the resulting bipartite oriented graph $D^*$ has minimum semidegree at least $3k/8+2$. Then by assumption, one can see that $D^*$ contains a copy of $P(s,t)$ for any $s,t$ with $s+t=k$, which follows that there is a $Q(s,t)$ in $D$, as desired. In the following, we will prove that $D$ indeed admits a copy of $P(s,t)$ for any $s,t$ with $s+t=k$.   

By symmetry, we may assume
without loss of generality that $s\ge t$. Hence, we infer that
$s\ge\lceil k/2\rceil$ and $t\le\lfloor k/2\rfloor$. Moreover, we can assume $k\geq 5$ since otherwise we are done by
Observation~\ref{lem:greedy}. Suppose to the contrary that there exists no $P(s,t)$ for some $s,t$ with $s\geq t$ and $s+t=k$. By using Lemma~\ref{longpc}, one can assume that $t\geq 1$.
It is convenient to denote $d:=\delta^0(D)\ge\frac{3k}{8}+2.$
In particular, $d>3k/8$.  Let
$P=v_1v_2\ldots v_\alpha$ be a longest directed path in $D$.  By
Observation~\ref{obs:endpoints},
$        N^-(v_1)\cup N^+(v_\alpha)\subseteq V(P).
$
By symmetry, we may assume without loss of generality that $v_1\in V_1$.  To proceed with the proof, we will divide the discussion into two parts based on the parity of the length of $P$.

\begin{center}
\textbf{Part I. The length of $P$ is odd  }
\end{center} 

\noindent

It is easily seen that $\alpha$ is even and $v_{\alpha}\in V_2$. Next, we would like to consider two cases as below.

\medskip
\noindent\textbf{Case 1.} $k/2\le s\le 2(k-1)/3$.
\medskip

Since $k/2\le s\le2(k-1)/3<2k/3$, we have $2k-s\leq 3k/2$.  Lemma~\ref{longpc} and the  hypothesis on $d$ give
\[
        \alpha\ge4d-1\ge\frac{3k}{2}+7\ge2k-s.
\]
Moreover, since $t\le k/2$, 
\[
        d\ge\frac{3k}{8}+2\ge\frac{t}{2}+1 \qquad \text{and}\qquad \alpha\ge2k-s\ge2k-2k/3\ge  2s+1.
\]
 Let $X=\{v_1,v_2,\dots, v_{t}\}$, $Y=\{v_{t+1},v_{t+2},\dots,v_{\alpha-t}\}$ and $Z=\{v_{\alpha-t+1},v_{\alpha-t+2}\dots,v_{\alpha}\}$. Notice that $v_i\in V_2$  if $ i\in[\alpha]$ is even.

If $N^-(v_1)\cap Y\ne\emptyset$ or $N^+(v_\alpha)\cap Y\ne\emptyset,$ then we can obtain, by Lemma~\ref{lemma} $(i)$, that there is a copy of $P(s,t)$, which leads to a contradiction.  Moreover, if $v_1\in N^+(v_\alpha)$, then $D[V(P)]$ contains a hamiltonian directed cycle. By Lemma~\ref{lemma} $(iii)$ we can also find a copy of $P(s,t)$, a contradiction. Hence, we may assume that $N^-(v_1)\subseteq X\cup(Z\backslash\{ v_\alpha\})$ and $N^+(v_\alpha)\subseteq (X\backslash\{v_1\})\cup Z$.

\begin{claim}\label{V_i}
There exists an even index
$
 i'\in
 \left[
   \alpha-k+\left\lceil\frac{3s}{2}\right\rceil,
   \alpha-1
 \right]
$
such that $v_{i'}\in N^-(v_1)$.
\end{claim}

\begin{proof}
Suppose for the sake of contradiction that the assertion does not hold. Since
$
N^-(v_1)\subseteq X\cup(Z\backslash\{v_\alpha\}),
$
the indices of all in-neighbors of 
$v_1$ are even and belong to
\[
[2,t]\cup
\left[
 \alpha-t+1,\,
 \alpha-k+\left\lceil\frac{3s}{2}\right\rceil-1
\right].
\]
Observe that the second interval  has at most $ \left\lceil\frac{s}{2}\right\rceil-1$ indices. Recall that $s\ge k/2$, we thus have 
$$
d^-(v_1)
\le
 \left\lfloor\frac{t}{2}\right\rfloor
 +\left\lceil
   \frac{\lceil s/2\rceil-1}{2}
  \right\rceil\\
\le \frac{t}{2}+\frac{s}{4}+1\\
=\frac{k}{2}-\frac{s}{4}+1\\
\le\frac{3k}{8}+1
<d.$$
But this is impossible because $d^-(v_1)\geq \delta^0(D)=d$.
\end{proof}

By Lemma~\ref{lemma} $(ii)$, we can assume that
$N^-(v_{i'+1})\subseteq V(P)$. Suppose first that $N^-(v_{i'+1})\cap Y\cap V_2=\emptyset.$ It is not difficult to deduce that $|N^+(v_\alpha)\cap Z\cap V_1|\le \lceil \frac{t}{2}\rceil-1$ and $|N^-(v_{i'+1})\cap Z\cap V_2|\le \lceil \frac{t}{2}\rceil-1$. This indicates that
$$
|N^-(v_{i'+1})\cap X\cap V_2|
\ge d^-(v_{i'+1})-|Z\cap V_2|\\
\ge d-\left\lceil\frac{t}{2}\right\rceil
>\frac{t}{4}
=\frac{|X|}{4}.
$$
Similarly, one can also get
$$
|N^+(v_\alpha)\cap X\cap V_1|
\ge d^+(v_\alpha)-|Z\cap V_1|\\
\ge d-\left\lfloor\frac{t}{2}\right\rfloor
>\frac{|X|}{4}.$$

Now we define the set
 $$R=\{v_{j-1}:v_{j}\in N^+(v_{\alpha})\cap X\cap V_1\}.$$
It is obviously that $R\subseteq X\cap V_2$. Furthermore, since $v_{1}\notin N^+(v_\alpha)$, we obtain that $|R|=|N^+(v_\alpha)\cap X|$. Therefore, $|N^-(v_{i'+1})\cap X|$ and $|R|>|X|/4$ hold. According to the pigeonhole principle, we know that $R \cap (N^-(v_{i'+1})\cap X\cap V_2)\ne\emptyset$. Thus, there must exist  a vertex $v_h\in N^-(v_{i'+1})\cap X\cap V_2$ such that $v_{h+1}\in  N^+(v_{\alpha})\cap X\cap V_1$.  
Then
$$\langle v_1,v_2,\ldots ,v_h,v_{i'+1},\ldots ,v_\alpha,
v_{h+1},v_{h+2},\ldots ,v_{i'}\rangle
$$
is a hamiltonian directed cycle of $D[V(P)]$. By applying  Lemma~\ref{lemma} $(iii)$, we are able to find a copy of $P(s,t)$, which contradicts the assumption that there is no $P(s,t)$.

Thereby, we may assume that $ N^-(v_{i'+1})\cap Y\cap V_2\ne\emptyset$. Choose arbitrarily a vertex
$
v_h\in N^-(v_{i'+1})\cap Y\cap V_2.
$
  We will consider the following two subcases, i.e.,
$(a_1)$. $i'\ge h+s$; $(a_2)$. $i'\le h+s-1$.

For $(a_1)$, we denote by $$P_1=v_1v_2\dots v_{h}v_{i'+1}  \qquad \text{and} \qquad P_2=v_{h+1}v_{h+2}\dots v_{i'+1}.$$ It is easily seen that $\ell(P_1)\ge t$ and $\ell(P_2)=i'-h\ge s$ by the assumption that $i'\ge h+s$. By Observation~\ref{sub-path}, we know that $P_1+ P_2$ contains a copy of $P(s,t)$, a contradiction .

For $(a_2)$, we pick a vertex $v_{j'}\in N^+(v_\alpha)\cap V_1$ with $j'\in\{\lfloor s/2\rfloor+1,\lfloor s/2\rfloor+2,\dots,t\}$. The existence of such a vertex follows by a similar argument as that of Claim~\ref{V_i}. Indeed, if there is no such a vertex, then $N^+(v_\alpha)\subseteq((Z\backslash\{v_\alpha\})\cup  \{v_1,v_2,\dots, \lfloor s/2\rfloor\})\cap V_1$. 
So we have $$d^+(v_\alpha)\le \lceil t/2\rceil -1+ \lfloor s/4\rfloor\le k/2- \lceil s/4\rceil<\delta^0(D),$$ a contradiction.

Let $$P_1=v_{h+1}v_{h+2}\dots v_{i'} v_1 v_2\dots v_{j'} \qquad \text{and} \qquad P_2=v_{j'+1}v_{j'+2}\dots v_{h}v_{i'+1}v_{i'+2}\dots v_{\alpha}v_{j'}.$$ Since $j'\ge\lfloor s/2\rfloor+1 $, $i'\ge \alpha-k+\lceil 3s/2\rceil$ as well as $h\le \alpha-t$,  we have \[
i'-h-1\ge
\left\lceil\frac{s}{2}\right\rceil-1.
\]  Then $\ell(P_1)=i' - h - 1+j'\ge s$. On the other hand, 
since $i'\le h+s-1$, we can obtain that $h\ge i'-s+1$. As a consequence, we get
$$
\ell(P_2)
=\alpha-i'+h-j'\\
\ge
(2k-s)-i'+(i'-s+1)-t\\
=2k-2s-t+1\\
=t+1>t.
$$

This means that $P_1+ P_2$ contains a copy of $P(s,t)$ by Observation~\ref{sub-path}, leading to a contradiction.

\medskip
\noindent\textbf{Case 2. $s>2(k-1)/3$.}
\medskip

Recall that $d=\delta^0(D)\ge3k/8+2$ and $\alpha\ge4d-1\ge3k/2+7$ by Lemma~\ref{longpc}. Define the sets
\[
 J=\{v_{k-s+1},v_{k-s+2},\ldots,v_{\alpha-s}\}\qquad \text{and}
 \qquad
 K=\{v_{s+1},v_{s+2},\ldots,v_{\alpha-k+s}\}.
\]
Note that $|J|=|K|=\alpha-k$.  Moreover, $J\cap K\ne\emptyset$ if and only if
$\alpha\ge2s+1$, and in this case
\[
 J\cup K=\{v_{k-s+1},v_{k-s+2},\ldots,v_{\alpha-k+s}\}.
\]

Next, we will show that $N^+(v_\alpha)\cap(J\cup K)\cap V_1\ne\emptyset$. Suppose, on the contrary, that this is not the case. Then $N^+(v_\alpha)\cap(J\cup K)\cap V_1=\emptyset$. It can be derived from  Observation~\ref{obs:endpoints} that $N^+(v_\alpha)\subseteq V(P)\cap V_1$, implying that $N^+(v_\alpha)\subseteq (V(P)\backslash(J\cup K))\cap V_1$. Combining this with $\alpha\ge 3k/2+7$, we can obtain that if $J\cap K=\emptyset$, then the two intervals contain exactly $\alpha-k$ odd indices. Hence
$$
d^+(v_\alpha)\le \frac{\alpha}{2}-(\alpha-k)\\ \notag
= k-\frac{\alpha}{2}
\le \frac{k}{4}-\frac{7}{2},
$$
which contradicts that $\delta^0(D)>3k/8$. Additionally, if $J\cap K\ne\emptyset$, then $J\cup K$ contains exactly $(\alpha-2(k-s))/2$ odd indices. Therefore, we obtain that
$$
d^+(v_\alpha)\le\frac{\alpha}{2}-\frac{\alpha-2(k-s)}{2}\\ \notag
= k-s
\le \frac{k+1}{3}<\delta^0(D),$$
also a contradiction. This proves that $N^+(v_\alpha)\cap(J\cup K)\cap V_1\ne\emptyset$.

Let $v_i\in N^+(v_\alpha)\cap(J\cup K)\cap V_1$.  Denote by
\[
  P_1=v_1 v_{2}\dots v_{i}\qquad \text{and}
 \qquad
 P_2=v_{i+1}v_{i+2}\dots v_{\alpha} v_i,
\]
Clearly, $\ell(P_1)=i-1$ and $\ell(P_2)=\alpha-i$. If $v_i\in J$, then
$\ell(P_1)\ge k-s=t$ and $\ell(P_2)\ge s$. If $v_i\in K$, then
$\ell(P_1)\ge s$ and $\ell(P_2)\ge k-s=t$. In any case, we can derive from Observation~\ref{sub-path}that $P_1+ P_2$ contains a copy of $P(s,t)$, which leads to a contradiction.

This completes our proof of {\bf{Part I}}.

\begin{center}
\textbf{Part II. The length of $P$ is even}
\end{center} 

\noindent

By the assumption that $v_1\in V_1$, one easily sees that $v_\alpha\in V_1$. In particular, $\alpha$ is odd and $v_i\in V_2$  if and only if $ i$ is even. 
 By Lemma~\ref{longpc}, we have
\begin{equation}\label{eq:alpha-lower}
 \alpha\ge4d-1\ge\frac{3k}{2}+7.
\end{equation}
In particular, $\alpha\ge k+2$.

Set
\[
 R=\{v_{s+1},\ldots,v_{\alpha-t-1}\}\qquad \text{and}\qquad
 B=\{v_{t+1},\ldots,v_{\alpha-s-1}\}.
\]
It is easy to see that both sets are nonempty by \eqref{eq:alpha-lower}.

\begin{claim}\label{cl:middle-free}
$(i).$  $v_\alpha\nsim R$.
$(ii).$  $v_\alpha\nsim B$.
\end{claim}
\begin{proof}

We only prove $(i)$, and the proof of $(ii)$ follows by a nearly identical argument.
Suppose that there exists a vertex $v_i\in R$ satisfying that $v_i\sim v_\alpha$.  If $v_i\to v_\alpha$, then let
$$P_1=v_{i-s+1}v_{i-s+2}\dots v_{i} v_{\alpha}  \qquad \text{and} \qquad P_2=v_{\alpha-t}v_{\alpha-t+1}\dots v_{\alpha}.$$ One easily finds that $\ell(P_1)= s$ and $\ell(P_2)= t$. Therefore, $P_1$ and $P_2$ form a
$P(s,t)$, a contradiction.  If $v_\alpha\to v_i$, then let
$$P_1=v_{i-s}v_{i-s+1}\dots v_{i}   \qquad \text{and} \qquad P_2=v_{\alpha-t+1}v_{\alpha-t+2}\dots v_{\alpha}v_i.$$  Thus, $P_1$ and $P_2$ form a
$P(s,t)$, also a  contradiction.
\end{proof}

\begin{claim}\label{cl:no-common-end}
There exists no vertex 
$v$ for which there are two longest directed paths $P_1$ and $P_2$ such that 
$v$ is the initial vertex of $P_1$ and the terminal vertex of $P_2$, and $V(P_1)=V(P_2)=V(P)$.
\end{claim}
\begin{proof}
Suppose $v$ is an initial vertex of $P_1$ and a terminal vertex of $P_2$. Let $P_2=v_1v_2\ldots v_{\alpha-1}v$, and $R=\{v_{s+1},\ldots,v_{\alpha-t-1}\}$, $B=\{v_{t+1},\ldots,v_{\alpha-s-1}\}$ as we defined before. From Observation~\ref{obs:endpoints} we can derive that
$N^+(v)\cup N^-(v)\subseteq V(P)$.  By Claim~\ref{cl:middle-free}, every vertex in $R\cup B$ is nonadjacent to $v$.
  Since $D$ is a bipartite oriented graph, there exist at least
$\lfloor(\alpha-k-1)/2\rfloor$ vertices belonging to different part from where $v$ is located which are nonadjacent to $v$.  Hence, we get
\[
 d^+(v)+d^-(v)
 \le \frac{\alpha-1}{2}-\left\lfloor\frac{\alpha-k-1}{2}\right\rfloor
 \le \left\lceil\frac{k}{2}\right\rceil.
\]
On the other hand,
$d^+(v)+d^-(v)\ge2d>3k/4>\lceil k/2\rceil$ for $k\ge5$. This leads to a contradiction.
\end{proof}

\begin{claim}\label{cl:boundary-arcs}
$v_\alpha\nrightarrow v_{\alpha-t}$ and
$v_\alpha\nrightarrow v_{\alpha-s}$.
\end{claim}
\begin{proof}
If $v_\alpha\to v_{\alpha-t}$, then let
$$P_1=v_{\alpha-k}v_{\alpha-k+1}\dots v_{\alpha-t}  \qquad \text{and} \qquad P_2=v_{\alpha-t+1}v_{\alpha-t+2}\dots v_{\alpha}v_{\alpha-t}.$$ Clearly, one can check that $\ell(P_1)=(\alpha-t)-(\alpha-k)=s$ and $\ell(P_2)= t$, implying that $P_1$ and $P_2$ form a
$P(s,t)$, a contradiction.

 If $v_\alpha\to v_{\alpha-s}$, then let
 $$P_1=v_{\alpha-s+1}v_{\alpha-s+2}\dots v_{\alpha} v_{\alpha-s} \qquad \text{and} \qquad P_2=v_{\alpha-k}v_{\alpha-k+1}\dots v_{\alpha-s-1}v_{\alpha-s}.$$ It is easily seen  that $\ell(P_1)= s$ and $\ell(P_2)=(\alpha-s)-(\alpha-k)= t$. Therefore, $P_1$ and $P_2$ form a
$P(s,t)$, also a contradiction.
\end{proof}

Let
\[
 \gamma:=\min\{i:v_\alpha\to v_i\}\qquad  \text{and} \qquad
 \xi:=\max\{i:v_i\to v_1\}.
\]
The minimum semidegree condition yields the existence of the two indices, and both are even.
By Claims~\ref{cl:middle-free} and \ref{cl:boundary-arcs}, if $\gamma>s$, then
all out-neighbors of $v_\alpha$ on $P$ are contained in the set of even-indexed vertices of
$\{v_{\alpha-t+1},\ldots,v_{\alpha-2}\}$, and hence
$d^+(v_\alpha)\le\lfloor t/2\rfloor<d$, a contradiction.  Thus
\begin{equation}\label{eq:gamma-range}
        2\le\gamma\le s,
        \qquad \xi\le\alpha-1.
\end{equation}

\begin{claim}\label{cl:R-free-v1}
$N^-(v_1)\cap(R\cup\{v_{\alpha-t}\})=\emptyset$.
\end{claim}
\begin{proof}
If $v_i\to v_1$ for some $v_i\in R\cup\{v_{\alpha-t}\}$, then let
$$P_1=v_{i-s+\gamma}v_{i-s+\gamma+1}\dots v_{i} v_1v_2\dots v_{\gamma} \qquad \text{and} \qquad P_2=v_{\alpha-t+1}v_{\alpha-t+2}\dots v_{\alpha} v_{\gamma}.$$
Therefore,  $P_1$ and $P_2$ form a
$P(s,t)$, a contradiction.
\end{proof}

\begin{claim}\label{cl:switch}
Suppose that $\gamma<q<\xi$, $q$ is even, and
$v_\alpha\to v_q$ and $v_{q-2}\to v_1$.  Then there are two longest directed
paths with vertex set $V(P)$ such that $v_{q-1}$ is the initial vertex of one
and the terminal vertex of the other.
\end{claim}
\begin{proof}
Define
\[
\begin{split}
 L_q={}&v_{q-1}v_q\ldots v_\alpha v_\gamma v_{\gamma+1}\ldots
        v_{q-2}v_1v_2\ldots v_{\gamma-1},\\
 M_q={}&v_{\xi+1}v_{\xi+2}\ldots v_\alpha v_qv_{q+1}\ldots
        v_\xi v_1v_2\ldots v_{q-1}.
\end{split}
\]
Both $L_q$ and $M_q$ are directed paths covering all vertices of $P$.
Moreover, $v_{q-1}$ is the initial vertex of $L_q$ and the terminal vertex of $M_q$.
\end{proof}

We proceed by considering cases according to the value of $\gamma$.

\medskip
\noindent\textbf{Case 1. $2\le\gamma\le t$.}

\begin{claim}\label{cl:B-free-v1}
$N^-(v_1)\cap(B\cup\{v_{\alpha-s}\})=\emptyset$.
\end{claim}
\begin{proof}
If $v_i\to v_1$ for some $v_i\in B\cup\{v_{\alpha-s}\}$, then let
$$P_1=v_{\alpha-s+1}v_{\alpha-s+2}\dots  v_{\alpha}v_{\gamma}  \qquad \text{and} \qquad P_2=v_{i-t+\gamma}v_{i-t+\gamma+1}\dots v_{i}v_1v_2\dots v_{\gamma}.$$ It is easy to find that $\ell(P_1)= s$ and $\ell(P_2)= t$. Therefore, $P_1$ and $P_2$ form a
$P(s,t)$, a contradiction.
\end{proof}

Since the set \{$v_1,\ldots,v_t$\} contains at most $\lfloor t/2\rfloor<d$ even-indexed vertices, we have $\xi>t$.  Claims~\ref{cl:R-free-v1} and
\ref{cl:B-free-v1} therefore imply
\begin{equation}\label{eq:xi-caseA}
 \xi\in[\alpha-s+1,s]\cup[\alpha-t+1,\alpha-1].
\end{equation}
Put
\[
 U=\{v_1,\ldots,v_\gamma\},\quad
 W=\{v_\xi,\ldots,v_\alpha\},\quad
 X_1=\{v_{\gamma+1},\ldots,v_t\},\quad
 Y_1=\{v_{\alpha-t},\ldots,v_{\xi-1}\},
\]
and $F=\{v_{\alpha-s},\ldots,v_s\}$.

For an interval $I\subseteq[\alpha]$, write
\[
 e(I):=|\{i\in I:i\text{ is even}\}|.
\]
All vertices in $N^-(v_1)\cup N^+(v_\alpha)$ have even indices.  We shall use
the following parity-restricted shifting rule.  For a set $A$ of even indices,
put
\[
 A-2:=\{i-2:i\in A\}.
\]
The map $i\mapsto i-2$ is injective and preserves parity.   We note that at most one shifted even index
crosses such an endpoint.  Notice also that
\begin{equation}\label{eq:end-arcs}
 v_2\notin N^-(v_1)\qquad\text{and}\qquad
 v_{\alpha-1}\notin N^+(v_\alpha),
\end{equation}
for otherwise a directed $2$-cycle would occur.

\smallskip

\noindent\textbf{Subcase 1. $F=\emptyset$.}

Since $F=\emptyset$, we have $\xi\in[\alpha-t+1,\alpha-1]$.  We claim that there exists an even index $q$ with $v_\alpha\to v_q$ and $v_{q-2}\to v_1$,
for which either
\[
q\in X_1,\ q-2\in X_1\cup\{\gamma\}
\qquad\text{or}\qquad
q\in Y_1,\ q-2\in Y_1.
\]
Indeed, suppose that no such $q$ exists.  Let
\[
 A_X=\{i:v_i\in N^+(v_\alpha)\cap X_1\},\qquad
 A_Y=\{i:v_i\in N^+(v_\alpha)\cap Y_1\},
\]
and put $x=|A_X|$ and $y=|A_Y|$.

For every $i\in A_X$, the index $i-2$ is even and belongs to $X_1\cup\{\gamma\}\subseteq[2,t]$. Since $v_\alpha\to v_i$, the assumption that no suitable $q$ exists implies that $v_{i-2}\notin N^-(v_1)$. Hence the $x$ distinct indices in $A_X-2$ exclude $x$ potential in-neighbors of $v_1$ among the even-indexed vertices of $[2,t]$. Consequently,
	
	$$
	\big|N^-(v_1)\cap (U\cup X_1)\cap V_2\big|
	\le e([2,t])-x.
	$$

For $A_Y$, all but possibly the first even index of $Y_1$ remain in $Y_1$ after the shift $i\mapsto i-2$. Since $v_\alpha\to v_i$ for every $i\in A_Y$, the assumption that no suitable $q$ exists implies that $v_{i-2}\notin N^-(v_1)$ for every shifted index $i-2$ that remains in $Y_1$. Let
	
	$$
	\eta_Y:=|(A_Y-2)\setminus[\alpha-t,\xi]|\in\{0,1\}.
	$$
	
	If $\eta_Y=1$, then necessarily $t$ is even, the exceptional index is $\alpha-t+1$, and its shift is $\alpha-t-1\in R$. Hence the $y-\eta_Y$ shifted indices of $A_Y$ lying in $[\alpha-t,\xi]$ exclude $y-\eta_Y$ potential in-neighbors of $v_1$. Consequently,
	
	$$
	\big|N^-(v_1)\cap (Y_1\cup\{v_\xi\})\cap V_2\big|
	\le e([\alpha-t,\xi])-y+\eta_Y.
	$$
	
	By the maximality of $\xi$ and Claims~\ref{cl:R-free-v1} and \ref{cl:B-free-v1}, every in-neighbor of $v_1$ lies in
	
	$$
	(U\cup X_1)\cup(Y_1\cup\{v_\xi\}).
	$$
	
	Therefore,	
	\begin{align}\notag
		d^-(v_1)&=|N^-(v_1)\cap (U\cup X_1)\cap V_2|+ |N^-(v_1)\cap (Y_1\cup \{v_\xi\})\cap V_2|\\
		&\le e([2,t])-x+e([\alpha-t,\xi])-y+\eta_Y\notag\\
		&\le e([2,t])-x+e([\alpha-t,\xi])-y+1.\label{eq:A1-in}
	\end{align}

By the minimality of $\gamma$ and Claim~\ref{cl:middle-free}, every
out-neighbor of $v_\alpha$ lies in
$\{v_\gamma\}\cup X_1\cup Y_1\cup W$.  The vertex $v_{\alpha-1}$ is not an
out-neighbor by \eqref{eq:end-arcs}.  Therefore
\begin{align} \notag
 d^+(v_\alpha)&=|N^+(v_\alpha)\cap (X_1\cup\{v_\gamma\})\cap V_2|+ |N^+(v_\alpha)\cap Y_1|+|N^+(v_\alpha)\cap W\cap V_2|\\
 &\le 1+x+y+e([\xi,\alpha])-1\notag\\
 &=x+y+e([\xi,\alpha]).\label{eq:A1-out}
\end{align}
Combining \eqref{eq:A1-in} and \eqref{eq:A1-out}, and recalling that $\xi$ is even, 
we obtain
\begin{align}
 d^-(v_1)+d^+(v_\alpha)
 &\le e([2,t])+e([\alpha-t,\xi])+e([\xi,\alpha])+1\notag\\
 &=e([2,t])+e([\alpha-t,\alpha])+2\notag\\
 &=t+2.\label{eq:A1-count}
\end{align}
On the other hand,
\[
 d^-(v_1)+d^+(v_\alpha)\ge2d\ge\frac{3k}{4}+4,
\]
whereas $t+2\le\lfloor k/2\rfloor+2<3k/4+4$, which leads to a contradiction.  Therefore, there exists $q$ satisfying that  $\gamma<q<\xi$, $q$ is even and
$v_\alpha\to v_q$ and $v_{q-2}\to v_1$.  By Claim~\ref{cl:switch}, there are two longest directed
paths with vertex set $V(P)$ such that $v_{q-1}$ is the initial vertex of one
and the terminal vertex of the other, contradicting Claim~\ref{cl:no-common-end}.

\smallskip

\noindent\textbf{Subcase 2. $F\ne\emptyset$.}
\smallskip

\noindent\emph{Subcase 2.1.}
$\alpha-s+1\le\xi\le s$.
\smallskip

Put
\[
 F_1=\{v_{\alpha-s},\ldots,v_{\xi-1}\}.
\]
We claim that there is an even index $q$ such that either
\[
 q\in X_1,\ q-2\in X_1\cup\{\gamma\},
 \qquad\text{or}\qquad
 q\in F_1,\ q-2\in F_1,
\]
and $v_\alpha\to v_q$, $v_{q-2}\to v_1$.  Indeed, suppose that no such
$q$ exists.  Put
\[
 A_X=\{i:v_i\in N^+(v_\alpha)\cap X_1\},\qquad
 A_F=\{i:v_i\in N^+(v_\alpha)\cap F_1\},
\]
and let $x=|A_X|$ and $f=|A_F|$.

For every $i\in A_X$, $i-2$ is an even index contained in $[2,t]$ with $v_{i-2}\notin N^-(v_1)$.
Hence the set $A_X-2 = \{i-2\mid i\in A_X\}$ consists of $x$ distinct indices, which exclude $x$ potential in‑neighbors of $v_1$ among the even‑indexed vertices of $[2,t]$.

For $A_F$, all but possibly the first even index of $F_1$ shift into
$[\alpha-s,\xi]$.  Define
\[
 \eta_F:=|(A_F-2)\backslash[\alpha-s,\xi]|\in\{0,1\}.
\]
If $i-2\in[\alpha-s,\xi]$ for some $i\in A_F$, we have
$i-2\in F_1$. Hence the assumption that no suitable $q$ exists gives
$v_{i-2}\notin N^-(v_1)$.  By the maximality of $\xi$ and
Claim~\ref{cl:B-free-v1}, all in-neighbors of $v_1$ lie in
$(U\cup X_1)\cup(F_1\cup\{v_\xi\})$.  Therefore
\begin{align}
 d^-(v_1)
 &=|N^-(v_1)\cap(U\cup X_1)\cap V_2|
   +|N^-(v_1)\cap(F_1\cup\{v_\xi\})\cap V_2|\notag\\
 &\le e([2,t])-x+e([\alpha-s,\xi])-f+\eta_F\notag\\
 &\le e([2,t])-x+e([\alpha-s,\xi])-f+1.
 \label{eq:A21-in}
\end{align}

By the minimality of $\gamma$ and Claim~\ref{cl:middle-free}, every
out-neighbor of $v_\alpha$ lies in
\[
 \{v_\gamma\}\cup X_1\cup F_1
 \cup\{v_\xi,\ldots,v_s\}
 \cup\{v_{\alpha-t},\ldots,v_\alpha\}.
\]
Moreover, $v_{\alpha-1}\notin N^+(v_\alpha)$ by \eqref{eq:end-arcs}.
Consequently,
\begin{align}
 d^+(v_\alpha)&=|N^+(v_\alpha)\cap( X_1\cup \{ v_\gamma \})\cap V_2| + |N^+(v_\alpha)\cap( F_1\cup\{v_\xi,\ldots,v_s\})\cap V_2| \notag\\  
& \quad +|N^+(v_\alpha)\cap\{v_{\alpha-t},\ldots,v_\alpha\}\cap V_2|\notag\\
 &\le 1+x+f+e([\xi,s])+e([\alpha-t,\alpha])-1\notag\\
 &=x+f+e([\xi,s])+e([\alpha-t,\alpha]).
 \label{eq:A21-out}
\end{align}
Combining with \eqref{eq:A21-in} and \eqref{eq:A21-out}, using the fact that $\xi$ is even
and $\eta_F\le1$, we have
\begin{align}
 d^-(v_1)+d^+(v_\alpha)
 &\le e([2,t])+e([\alpha-s,\xi])+e([\xi,s])
      +e([\alpha-t,\alpha])+1\notag\\
 &=e([2,t])+e([\alpha-s,s])+e([\alpha-t,\alpha])+2\notag\\
 &=t+\frac{2s-\alpha+1}{2}+2\notag\\
 &=\frac{2k-\alpha+5}{2}.
 \label{eq:A21-count}
\end{align}
By \eqref{eq:alpha-lower}, the last quantity is at most $k/4-1$, whereas
\[
 d^-(v_1)+d^+(v_\alpha)\ge2d\ge\frac{3k}{4}+4,
\]
a contradiction.  Hence such a $q$ exists.  Since $\gamma<q<\xi$ in both
alternatives, Claim~\ref{cl:no-common-end} yields the desired contradiction.

\smallskip
\noindent\emph{Subcase 2.2.}
$\alpha-t+1\le\xi\le\alpha-1$.

We claim that there is an even index
$q$ such that either
$$
q\in X_1,\ q-2\in X_1\cup\{\gamma\},
\qquad\text{or}\qquad
q\in F,\ q-2\in F,
\qquad\text{or}\qquad
q\in Y_1,\ q-2\in Y_1,
$$
and $v_\alpha\to v_q$, $v_{q-2}\to v_1$. Indeed, suppose that no such $q$ exists.
Put
\[
\begin{split}
 A_X&=\{i:v_i\in N^+(v_\alpha)\cap X_1\},\\
 A_F&=\{i:v_i\in N^+(v_\alpha)\cap F\},\\
 A_Y&=\{i:v_i\in N^+(v_\alpha)\cap Y_1\},
\end{split}
\]
and let $x=|A_X|$, $f=|A_F|$, and $y=|A_Y|$.

The shifts of all indices in $A_X$ remain in $[2,t]$.  For each of $A_F$
and $A_Y$, at most its first even index shifts across the left endpoint of
the corresponding interval.  
  Let
\[
 \eta_Y:=|(A_Y-2)\backslash[\alpha-t,\xi]|\in\{0,1\},  \qquad \eta_F:=
|(A_F-2)\backslash[\alpha-s,s]|
\in\{0,1\}.
\]
Consequently, by the maximality of $\xi$, and Claims~\ref{cl:R-free-v1} and
\ref{cl:B-free-v1},
\begin{align}\notag 
 d^-(v_1)& =|N^-(v_1)\cap (U\cup X_1)\cap V_2|+ |N^-(v_1)\cap F \cap V_2|+ |N^-(v_1)\cap (Y_1\cup \{v_{\xi}\})\cap V_2|\\  \notag
&\le e([2,t])-x+e([\alpha-s,s])-f+\eta_F
      +e([\alpha-t,\xi])-y+\eta_Y \\ 
 &\le e([2,t])-x+e([\alpha-s,s])-f+1
      +e([\alpha-t,\xi])-y+1.\label{eq:A2-in}
\end{align}
As in Subcase~A1, Claim~\ref{cl:middle-free}, the minimality of $\gamma$, and
\eqref{eq:end-arcs} imply
\begin{align}
 d^+(v_\alpha)
 &=|N^+(v_\alpha)\cap (X_1\cup\{v_\gamma\})\cap V_2|
   +|N^+(v_\alpha)\cap F\cap V_2| +|N^+(v_\alpha)\cap Y_1\cap V_2|\notag\\
 &\quad
   +|N^+(v_\alpha)\cap W\cap V_2|\notag\\
 &\le 1+x+f+y+e([\xi,\alpha])-1\notag\\
 &=x+f+y+e([\xi,\alpha]).
 \label{eq:A2-out}
\end{align}

Adding \eqref{eq:A2-in} and
\eqref{eq:A2-out}, we obtain
\begin{align}
 d^-(v_1)+d^+(v_\alpha)
 &\le e([2,t])+e([\alpha-s,s])
      +e([\alpha-t,\xi])+e([\xi,\alpha])+2\\ \notag
 &=t+\frac{2s-\alpha+1}{2}+3\notag\\
 &=\frac{2k-\alpha+7}{2}.\label{eq:A2-count}
\end{align}
By \eqref{eq:alpha-lower}, the last quantity is at most $k/4$, whereas the
left-hand side is at least $2d\ge3k/4+4$, a contradiction.  Therefore a
suitable $q$ exists, and Claim~\ref{cl:no-common-end} again yields a contradiction.

\medskip
\noindent\textbf{Case 2. $t<\gamma\le s$.}

\begin{claim}\label{cl:xi-gap}
$\xi\ge\gamma+4$.
\end{claim}
\begin{proof}
By the minimality of $\gamma$, Claim~\ref{cl:middle-free},
Claim~\ref{cl:boundary-arcs}, and \eqref{eq:end-arcs}, every out-neighbor of
$v_\alpha$ lies among the even-indexed vertices of
$[\gamma,s]\cup[\alpha-t,\alpha]$. Note that $v_{\alpha-1}\notin\ N^+(v_{\alpha})$. Hence
\begin{align}
 d^+(v_\alpha)
 &\le e([\gamma,s])+e([\alpha-t,\alpha])-1\notag\\
 &\le\left\lceil\frac{k-\gamma}{2}\right\rceil.\label{eq:xi-gap-out}
\end{align}
If $\xi\le\gamma+2$, then $\xi$ is even and $v_2\notin N^-(v_1)$, so
$d^-(v_1)\le\gamma/2$.  Therefore
\[
 d^-(v_1)+d^+(v_\alpha)
 \le\frac{\gamma}{2}+\left\lceil\frac{k-\gamma}{2}\right\rceil
 =\left\lceil\frac{k}{2}\right\rceil
 <\frac{3k}{4}+4\le2d,
\]
a contradiction.
\end{proof}

Together with Claim~\ref{cl:R-free-v1}, Claim~\ref{cl:xi-gap} implies
\[
 \xi\in[\gamma+4,s]\cup[\alpha-t+1,\alpha-1].
\]

\smallskip
\noindent\textbf{Subcase 1. $\gamma+4\le\xi\le s$.}

Let $X=\{v_{\gamma+1},\ldots,v_{\xi-1}\}$ and
$Y=\{v_{\alpha-t},\ldots,v_\alpha\}$.  Suppose that there is no even
$q\in X$ with $q-2\in X\cup\{\gamma\}$,
$v_\alpha\to v_q$, and $v_{q-2}\to v_1$.  Set
\[
 A_X=\{i:v_i\in N^+(v_\alpha)\cap X\},\qquad x=|A_X|.
\]
Every index in $A_X-2$ lies in $[2,\xi]$ and is excluded from the
in-neighborhood of $v_1$.  Since $\xi$ is the largest index of an
in-neighbor of $v_1$, we have
\begin{align}\notag
 d^-(v_1)&= |N^-(v_1)\cap (U\cup X\cup\{v_{\xi}\})\cap V_2| \\ 
& \le e([2,\xi])-x.\label{eq:B1-in} 
\end{align} 
By Claim~\ref{cl:middle-free}, all out-neighbors of $v_\alpha$ occur in
$\{v_\gamma\}\cup X\cup\{v_\xi,\ldots,v_s\}\cup Y$.  Since
$v_{\alpha-1}\notin N^+(v_\alpha)$,
\begin{align}\notag
 d^+(v_\alpha)&=|N^+(v_\alpha)\cap (X\cup\{v_\gamma\})\cap V_2|+ |N^+(v_\alpha)\cap Y\cap V_2|+ |N^+(v_\alpha)\cap\{v_\xi,\ldots,v_s\} \cap V_2|\\ \notag
 &\le1+x+e([\xi,s])+e([\alpha-t,\alpha])-1\notag\\
 &=x+e([\xi,s])+e([\alpha-t,\alpha]).\label{eq:B1-out}
\end{align}
Adding \eqref{eq:B1-in} and \eqref{eq:B1-out}, and noting that $\xi$ is
even, gives
\begin{align}
 d^-(v_1)+d^+(v_\alpha)
 &\le e([2,\xi])+e([\xi,s])+e([\alpha-t,\alpha])\notag\\
 &=e([2,s])+1+e([\alpha-t,\alpha])\notag\\
 &\le\left\lceil\frac{k}{2}\right\rceil+1
 <\frac{3k}{4}+4\le2d,
\end{align}
a contradiction.  Hence such a $q$ exists, which leads to a contradiciton with Claim~\ref{cl:no-common-end}.

\smallskip
\noindent\textbf{Subcase 2. $\alpha-t+1\le\xi\le\alpha-1$.}

Let
$
 X_2=\{v_{\gamma+1},\ldots,v_s\},
 Y_2=\{v_{\alpha-t},\ldots,v_{\xi-1}\}.
$
Suppose that there is no even index $q$ satisfying
$
v_\alpha\to v_q$ and $
v_{q-2}\to v_1,
$
such that either
\[
q\in X_2\quad\text{and}\quad q-2\in X_2\cup\{\gamma\},
\]
or
\[
q\in Y_2\quad\text{and}\quad q-2\in Y_2.
\]  Put
\[
 A_X=\{i:v_i\in N^+(v_\alpha)\cap X_2\},\qquad
 A_Y=\{i:v_i\in N^+(v_\alpha)\cap Y_2\},
\]
and $x=|A_X|$, $y=|A_Y|$.

All shifts from $A_X$ remain in $[2,s]$.  For $A_Y$, at most one shifted
index crosses the left endpoint $\alpha-t$.  Thus, Claim~\ref{cl:R-free-v1} and the maximality of $\xi$ give
\begin{align}\notag
 d^-(v_1)&= |N^-(v_1)\cap (U\cup X_2)\cap V_2|+ |N^-(v_1)\cap (Y_2\cup \{v_{\xi}\})\cap V_2| \\   
 &\le e([2,s])-x+e([\alpha-t,\xi])-y+1.\label{eq:B2-in}
\end{align}
Furthermore, by Claim~\ref{cl:middle-free}, the minimality of $\gamma$, and
\eqref{eq:end-arcs},
\begin{align}\notag
 d^+(v_\alpha)&=|N^+(v_\alpha)\cap (X_2\cup\{v_\gamma\})\cap V_2|+ |N^+(v_\alpha)\cap Y_2\cap V_2|+ |N^+(v_\alpha)\cap W\cap V_2|\\ \notag
 &\le1+x+y+e([\xi,\alpha])-1\notag\\
 &=x+y+e([\xi,\alpha]).\label{eq:B2-out}
\end{align}
Since $\xi$ is even, adding \eqref{eq:B2-in} and \eqref{eq:B2-out} yields
\begin{align}
 d^-(v_1)+d^+(v_\alpha)
 &\le e([2,s])+e([\alpha-t,\xi])+e([\xi,\alpha])+1\notag\\
 &=e([2,s])+e([\alpha-t,\alpha])+2\notag\\
 &\le\left\lceil\frac{k}{2}\right\rceil+2
 <\frac{3k}{4}+4\le2d,
\end{align}
a contradiction.  Therefore such a $q$ exists, and Claim~\ref{cl:no-common-end}
gives the final contradiction.

All cases lead to a contradiction.  Hence $D$ contains a copy of $P(s,t)$
when the two ends of the longest directed path belong to the same
bipartition class.  This completes the proof of {\bf {Part~II}}.

\vspace{0.2cm}

Combining the proofs of {\bf {Parts I and~II}}, we thus complete our proof of Theorem~\ref{A}.

\section{Concluding remarks}

 \noindent

In this paper, we prove that for any $k\geq 2$, every bipartite oriented graph $D$ with minimum semidegree at least $3k/8+2$ contains any oriented path with two blocks of length $k$, which can be regarded as an extension of some results on Conjecture~\ref{conj1} to the bipartite version. We believe that the lower bound is not tight, and it is very interesting to determine the corresponding threshold. More generally, we attempt to determine the minimum semidegree threshold for the existence of other oriented paths.

Jackson \cite{Jackson1981} proved that any oriented graph of minimum semidegree at least $k/2$ contains a directed path of length $k$, and Ayel~\cite{Ayel1982} showed that any bipartite oriented graph of minimum semidegree greater than $k/4$ contains a directed path of length $k$. 
Stein conjectured that every oriented graph of minimum semidegree greater than $k/2$ contains any oriented path of length $k$. Naturally, one may ask whether the threshold in the bipartite setting can be reduced by a factor of two. However, this is not the case. Let $G_{4,r}=G(X_1,Y_1,X_2,Y_2)$ be the $r$-regular bipartite tournament consisting of four independent sets $X_1,Y_1,X_2,Y_2$, each of which has size $r$, with all possible arcs directed cyclically as
\[
X_1\to Y_1,\qquad Y_1\to X_2,\qquad X_2\to Y_2,\qquad Y_2\to X_1.
\]
Observe that $\delta^0(G_{4,r})=r$. When $r=k/4$ for $k\equiv 0 \pmod 4$, one can see that $G_{4,k/4}$ contains no oriented path of length $k$. This indicates that the corresponding  minimum semidegree threshold in bipartite oriented graphs must be larger than $k/4$. Additionally, by taking $r=k/2$ for $k\equiv 0 \pmod 2$, we can see that there exists no anti-directed path of length $k$ in $G_{4,k/2}$ because the vertices of all anti-directed path must belong to $X_i\cup Y_j$ for some $i,j\in [2]$ with $i\neq j$. In general, the minimum semidegree threshold, in the bipartite setting, for forcing any oriented path with many blocks is closely related to the number of blocks. This leads to the following problem naturally.

\begin{problem}
Determine the smallest value $f(t)$ such that every bipartite oriented graph of minimum semidegree greater than $k/4+f(t)$ contains any oriented paths of length $k$ with $t$ blocks, where $k,t$ are integers with $t\leq k$.
\end{problem}

\subsection*{Acknowledgements}

 \noindent

The work was supported by the National Key R\&D Program of China (2023YFA1010203), the  Quantum Science and Technology--National Science and Technology Major Project (2021ZD0302902), and the National Natural Science Foundation of China (Nos. 12471336, 12501473).

\vskip 3mm
\end{spacing}
\end{document}